\documentclass[12pt]{amsart}
\usepackage{booktabs,graphicx, amsmath, amsfonts}
\usepackage{tabularx}
\usepackage{amsthm}
\usepackage{colortbl}
\usepackage[table]{xcolor}
\usepackage{rotating}
\usepackage{array}
\usepackage{geometry}
\usepackage{booktabs}
\definecolor{lightgray}{RGB}{200,200,200}
\definecolor{lightpink}{RGB}{255,182,193}
\definecolor{lightblue}{RGB}{180,220,255}
\definecolor{orange}{RGB}{255,165,0}
\renewcommand{\arraystretch}{}

\newtheorem{theorem}{Theorem}

\newtheorem{definition}{Definition}
\newtheorem{lemma}[theorem]{Lemma}

\title{New minor minimal non-apex graphs}
\author{Andrei Pavelescu, Elena Pavelescu, and Madeline Potter}

\address{\textit{andreipavelescu@southalabama.edu}, Department of Mathematics and Statistics, University of South Alabama, Mobile, AL  36688, USA}
\address{\textit{elenapavelescu@southalabama.edu}, Department of Mathematics and Statistics, University of South Alabama, Mobile, AL  36688, USA}
\address{\textit{mmp2123@jagmail.southalabama.edu}, University of South Alabama, Mobile, AL  36688, USA}

\begin{document}

\begin{abstract}
A graph is apex if it becomes planar after the deletion of one vertex. The family of apex graphs is closed under taking minors, so it is characterized by a finite set of forbidden minors. 
Determining the finite set of forbidden minors for apex graphs remains an open question.
In this paper, we list all forbidden minors for apex graphs with 12 or fewer vertices and all forbidden minors for apex graphs with 26 and fewer edges. We also present graphs outside of these ranges. 
We show that a graph with 13 vertices and minimal degree 6 is either apex or contains a $K_6$ minor, proving J\o rgensen's conjecture for order 13.
\end{abstract}


\maketitle
\section{Introduction}

An \textit{apex} graph is a graph that becomes planar after the removal of one vertex. 
The Graph Minor Theorem of Robertson and Seymour \cite{RT} implies that if a class of graphs is defined by a minor-closed property, then the set of minor minimal graphs without the property is finite. These graphs are called forbidden minors for the property. Planarity, the property of being apex, linkless embeddability, knotless embeddability are all minor-closed properties. Therefore, each of these classes of graphs has a finite set of forbidden minors.  
The set of forbidden minors for planarity is known to  contain exactly two graphs, the complete graph on six vertices $K_6$ and the complete bipartite graph $K_{3,3}$ \cite{Ku, Wa}. 
The set of forbidden minors for linkless embeddability is known to contain precisely seven graphs: $K_6$ and the graphs obtained from $K_6$ through sequences of delta-wye and wye-delta transformations \cite{CG, RST, Sa1}. This collection is also called the Petersen family of graphs, since it includes the Petersen graph.
The complete sets of forbidden minors for apexity and knotless embeddability are not yet known.  

Many researchers have searched for forbidden minors for apexity. 
These graphs are also known as minor minimal non-apex (MMNA) graphs because they are not apex and every one of their minors is apex. 
In \cite{JK}, the authors determine the 133 MMNA graphs that have connectivity two. 
Lipton et al. \cite{L} showed that apex obstructions have connectivity at most five. 
Pierce lists 157 MMNA graphs \cite{P}.
In this article, we present a complete list of MMNA graphs  with 12 or fewer vertices or 26 or fewer edges.
We also present several MMNA graphs outside of these ranges.
These additional graphs belong  to the delta-wye families of already found MMNA graphs with 11 or 12 vertices.
Although many of these graphs have been known previously, as far as we know, this list is more comprehensive than any other currently published.
In total, we present 339 graphs.
Our findings are computer-assisted. We conducted exhaustive searches of graphs with at most 12 vertices and of graphs with at most 26 edges. We present our findings in Section 2 and describe the computer programs in Section 3. For graphs with more than 12 vertices, the number of graphs is overwhelmingly large for us to apply our search methods. It is possible that these searches may be carried out in the future.
In Section 4, we prove Theorem 1, which describes restrictions that can be imposed on graphs with 13 vertices, potentially making such searches more efficient. This theorem strengthens J\o rgensen's conjecture that every 6-connected graph of order 13 without a $K_6$ minor is apex. The graph $K_6$ is MMNA; thus, any graph containing $K_6$ as a minor cannot be MMNA.

\begin{theorem} Let $G$ be a simple graph of order $13$ and minimal degree 6. Then either $G$ contains a $K_6$-minor or it is apex and isomorphic to $K_1\ast Ic$.
\label{main13}
\end{theorem}

Apex graphs provide a bridge between planar graphs and more complex graph families. Tools from planar graph such as separators and structural decompositions still apply. Their almost planar nature also gives them algorithmic advantages: many otherwise hard problems become tractable on apex graphs. As a result, they are both theoretically significant and useful for modeling real world networks that are planar except for one node.

\section{Graphs with 12 or fewer vertices or 26 or fewer edges}
%We begin by considering the vertex degrees and the number of edges of an MMNA graph.
We establish background and notation.
 A \textit{minor} of a graph is obtained through a sequence of vertex deletions, edge deletions, and edge contractions. An edge contraction identifies the endpoints of the edge and deletes any loops and multiple edges thus created. We denote the minimmum degree among the vertices of a graph $G$ by $\delta(G)$, and the maximum degree by $\Delta(G).$
%For a vertex $v \in V(G)$, the set $N_G(v)$ is the set of all vertices of $G$ that are adjacent to $v$, and $N_G[v]= N_G(v)\cup \{v\}$. 
A graph $G'$ is obtained from $G$ by a \textit{delta--wye} operation if a triangle $abc$ in $G$ is removed and a new vertex $v$ is added together with edges $va$, $vb$, and $vc$. 
This operation preserves the size of the graph and increases its order by one.

\begin{lemma}
A minor minimal non-apex graph has minimum vertex degree at least 3.
\label{lem:mindeg}
\end{lemma}
\begin{proof}
If a non-apex graph $G$ has a vertex $v$ of degree 0 or 1, then $G\setminus \{v\}$ is also non-apex.  
If a non-apex graph $G$ has a vertex $v$ of degree 2, then the graph $G/e$ obtained by contracting one of the two edges incident to $v$ is also non-apex. In either case, $G$ is not minor minimal non-apex.
\end{proof}
Lemma \ref{lem:mindeg} implies that $\lceil 3n/2 \rceil$ is a lower bound on the number of edges of $G$. On the other hand, by Mader \cite{Ma}, a graph of order $n\ge 6$ and at least $4n-9$ edges contains $K_6$ as minor. The graph $K_6$ is itself MMNA, so any other MMNA graph with $n$ vertices has at most $4n-10$ edges.

A graph is \textit{linklessly embeddable} if it can be embedded in space such that no two of its cycles are non-trivially linked. If not linklessly embeddable, a graph is said to be \textit{intrinsically linked}. 
The minor minimal intrinsically linked graphs, the seven graphs in the Petersen family, are MMNA. They are the graphs with 15 edges listed in Tables \ref{tab:mmna1} and \ref{tab:mmna2}.  
This means that all other MMNA graphs are necessarily linklessly embeddable.

Our search relies on previously known families of \textit{maximal linklessly embeddable graphs (maxnil)}. 
A graph is maxnil if it is linklessly emebddable and the addition of any one edge creates a graphs that is intrinsically linked. 
The MMNA graphs not in the Petersen family must appear as subgraphs of maxnil graphs of the same order. 
We used the complete list of maxnil graphs of 12 or fewer vertices to find the MMNA graphs of 12 or fewer vertices.

\begin{theorem}[\cite{NOPP, Pa}] There are a total of 846 maximal linklessly embeddable graphs with 11 or fewer vertices, of which 305 graphs are apex and 521 graphs are non-apex.
\end{theorem}

\begin{theorem}[\cite{LPP, Pa}] There are a total of 6,503 maximal linklessly embeddable graphs of order 12, of which 1,249 graphs are apex and 5,254 graphs are non-apex.
\end{theorem}

\renewcommand{\arraystretch}{1.6}
\begin{table}[htbp]
\centering
\scriptsize
\caption{Counts of MMNA graphs presented by Pierce in \cite{P}.}\vspace*{-0.1in}
\label{tab:mmna1}

\begin{tabularx}{\linewidth}{|c|*{17}{>{\centering\arraybackslash}X|}}
\hline
$n \,/\, e$ 
 & 15 & 16 & 17 & 18 & 19 & 20 & 21 & 22 & 23 & 24 & 25 & 26 & 27 & 28 & 29 & 30 & 31 \\
\hline
6  &  \cellcolor{lightblue}1 &  \cellcolor{lightgray}& \cellcolor{lightgray} & \cellcolor{lightgray} & \cellcolor{lightgray} & \cellcolor{lightgray} & \cellcolor{lightgray} & \cellcolor{lightgray} & \cellcolor{lightgray} & \cellcolor{lightgray} & \cellcolor{lightgray} & \cellcolor{lightgray} & \cellcolor{lightgray} & \cellcolor{lightgray} & \cellcolor{lightgray} & \cellcolor{lightgray} & \cellcolor{lightgray} \\
\hline
7  &  \cellcolor{lightblue}2 &  \cellcolor{lightgray}& \cellcolor{lightgray} & \cellcolor{lightgray} & \cellcolor{lightgray} & \cellcolor{lightgray} & \cellcolor{lightgray} & \cellcolor{lightgray} & \cellcolor{lightgray} & \cellcolor{lightgray} & \cellcolor{lightgray} & \cellcolor{lightgray} & \cellcolor{lightgray} & \cellcolor{lightgray} & \cellcolor{lightgray} & \cellcolor{lightgray} & \cellcolor{lightgray} \\
\hline
8  &   \cellcolor{lightblue}2 &  \cellcolor{lightgray}& \cellcolor{lightgray} & \cellcolor{lightgray} & \cellcolor{lightgray} & \cellcolor{lightgray}  & \cellcolor{lightblue}1 & \cellcolor{lightgray} & \cellcolor{lightgray} & \cellcolor{lightgray}  & \cellcolor{lightgray}& \cellcolor{lightgray} & \cellcolor{lightgray} & \cellcolor{lightgray}  & \cellcolor{lightgray} & \cellcolor{lightgray} & \cellcolor{lightgray} \\
\hline
9  &  \cellcolor{lightblue}1 & \cellcolor{lightgray} & \cellcolor{lightgray}& \cellcolor{lightblue}1 & \cellcolor{lightgray} & \cellcolor{lightblue}2 & \cellcolor{lightblue}1 &  \cellcolor{lightgray}  & \cellcolor{lightgray}&  \cellcolor{lightgray}  & \cellcolor{lightgray}&  \cellcolor{lightgray}  & \cellcolor{lightgray} &  \cellcolor{lightgray}  & \cellcolor{lightgray} & \cellcolor{lightgray} & \cellcolor{lightgray} \\
\hline
10 &  \cellcolor{lightblue}1 &  \cellcolor{lightgray}  & \cellcolor{lightgray}& \cellcolor{lightblue}2 & \cellcolor{lightblue}3 & \cellcolor{lightblue}7 & \cellcolor{lightblue}7 & \cellcolor{lightblue}4 & \cellcolor{lightblue}5 & \cellcolor{lightblue}2 &  \cellcolor{lightgray}  & \cellcolor{lightgray} &  \cellcolor{lightgray}  & \cellcolor{lightgray} &  \cellcolor{lightgray}  & \cellcolor{lightgray} &  \cellcolor{lightgray}\\
\hline
11 & \cellcolor{lightgray} &  \cellcolor{lightgray} & \cellcolor{lightgray} &  \cellcolor{lightgray} & \cellcolor{lightblue}2 & \cellcolor{lightblue}12 & \cellcolor{lightblue}14 & \cellcolor{orange}2 & \cellcolor{orange}4 & \cellcolor{orange}1 & \cellcolor{orange}1 &  & \cellcolor{orange}1 &  &  &  &\\
\hline
12 & \cellcolor{lightgray} & \cellcolor{lightgray} & \cellcolor{lightgray} & \cellcolor{lightblue}1 & \cellcolor{lightgray} & \cellcolor{lightblue}1 & \cellcolor{lightblue}12 & \cellcolor{orange}1 & \cellcolor{orange}6 & & \cellcolor{orange}2 & \cellcolor{orange}2 & \cellcolor{orange}2 & \cellcolor{orange}1 &  & \cellcolor{orange}1 &\\
\specialrule{1pt}{0pt}{0pt}
13 & \cellcolor{lightgray} & \cellcolor{lightgray} &  \cellcolor{lightgray} & \cellcolor{lightgray} &  \cellcolor{lightgray} & \cellcolor{lightgray} & \cellcolor{lightgray} & \cellcolor{orange}1 & \cellcolor{orange}3 &  & \cellcolor{orange}3 & \cellcolor{orange}3 & \cellcolor{orange}2 & \cellcolor{orange}1 & \cellcolor{orange}2 & \cellcolor{orange}1  & \\
\hline
14 & \cellcolor{lightgray} & \cellcolor{lightgray} &  \cellcolor{lightgray} & \cellcolor{lightgray} &  \cellcolor{lightgray} & \cellcolor{lightgray} & \cellcolor{lightgray} & & & \cellcolor{orange}1 & \cellcolor{orange}2 & \cellcolor{orange}2 & \cellcolor{orange}1 & \cellcolor{orange}7 & \cellcolor{orange}1 &  &\\
\hline
15 & \cellcolor{lightgray} & \cellcolor{lightgray} &  \cellcolor{lightgray} & \cellcolor{lightgray} &  \cellcolor{lightgray} & \cellcolor{lightgray} & \cellcolor{lightgray} & \cellcolor{lightgray} & &  &  &  & \cellcolor{orange}4 & \cellcolor{orange}8 & & & \\
\hline
16 & \cellcolor{lightgray} & \cellcolor{lightgray} &  \cellcolor{lightgray} & \cellcolor{lightgray} &  \cellcolor{lightgray} & \cellcolor{lightgray} & \cellcolor{lightgray} & \cellcolor{lightgray} & \cellcolor{lightgray}&  &  &  & \cellcolor{orange}1 & \cellcolor{orange}6 & & & \\
\hline
17 & \cellcolor{lightgray} & \cellcolor{lightgray} &  \cellcolor{lightgray} & \cellcolor{lightgray} &  \cellcolor{lightgray} & \cellcolor{lightgray} & \cellcolor{lightgray} & \cellcolor{lightgray} & \cellcolor{lightgray}&  \cellcolor{lightgray}&  \cellcolor{lightgray}&  & &  & & & \\
\hline
\end{tabularx}
\end{table}

\renewcommand{\arraystretch}{1.6}
\begin{table}[htbp]
\centering
\scriptsize
\caption{Counts of MMNA graphs with 18 or fewer vertices.}\vspace*{-0.1in}
\label{tab:mmna2}
\begin{tabularx}{\linewidth}{|c|*{20}{>{\centering\arraybackslash}X|}}
\hline
$n \,/\, e$ 
 & 15 & 16 & 17 & 18 & 19 & 20 & 21 & 22 & 23 & 24 & 25 & 26 & 27 & 28 & 29 & 30 & 31 & 32 & 33 & 34\\
\hline
6  &  \cellcolor{lightblue}1 &  \cellcolor{lightpink}& \cellcolor{lightpink}& \cellcolor{lightpink} & \cellcolor{lightpink} & \cellcolor{lightpink} & \cellcolor{lightpink} &\cellcolor{lightpink} & \cellcolor{lightpink}& \cellcolor{lightpink} & \cellcolor{lightpink} & \cellcolor{lightpink} & \cellcolor{lightpink} & \cellcolor{lightpink} & \cellcolor{lightpink} & \cellcolor{lightpink} & \cellcolor{lightpink} & \cellcolor{lightpink} & \cellcolor{lightpink}& \cellcolor{lightpink}
 \\
\hline
7  &  \cellcolor{lightblue}2 &  \cellcolor{lightgray}& \cellcolor{lightgray} & \cellcolor{lightgray} & \cellcolor{lightpink} & \cellcolor{lightpink} & \cellcolor{lightpink} & \cellcolor{lightpink} & \cellcolor{lightpink}& \cellcolor{lightpink} & \cellcolor{lightpink} & \cellcolor{lightpink} & 
\cellcolor{lightpink} & \cellcolor{lightpink} & \cellcolor{lightpink} & \cellcolor{lightpink} & \cellcolor{lightpink} & \cellcolor{lightpink} & \cellcolor{lightpink} & \cellcolor{lightpink}  \\
\hline
8  &   \cellcolor{lightblue}2 &  \cellcolor{lightgray}& \cellcolor{lightgray} & \cellcolor{lightgray} & \cellcolor{lightgray} & \cellcolor{lightgray}  & \cellcolor{lightblue}1 & \cellcolor{lightgray} & \cellcolor{lightpink} & \cellcolor{lightpink} & \cellcolor{lightpink}& \cellcolor{lightpink} & \cellcolor{lightpink} & \cellcolor{lightpink} & \cellcolor{lightpink} & \cellcolor{lightpink} &
\cellcolor{lightpink} & \cellcolor{lightpink}  & \cellcolor{lightpink} &\cellcolor{lightpink}\\
\hline
9  &  \cellcolor{lightblue}1 & \cellcolor{lightgray} & \cellcolor{lightgray}& \cellcolor{lightblue}1 & \cellcolor{lightgray} & \cellcolor{lightblue}2 & \cellcolor{lightblue}1 &  \cellcolor{lightgray}  & \cellcolor{lightgray}&  \cellcolor{lightgray}  & \cellcolor{lightgray}&  \cellcolor{lightgray}  & 
\cellcolor{lightpink} & \cellcolor{lightpink} &
\cellcolor{lightpink}  & \cellcolor{lightpink}   & \cellcolor{lightpink}   &  \cellcolor{lightpink}    & \cellcolor{lightpink}   & \cellcolor{lightpink}   \\
\hline
10 &  \cellcolor{lightblue}1 &  \cellcolor{lightgray}  & \cellcolor{lightgray}& \cellcolor{lightblue}2 & \cellcolor{lightblue}3 & \cellcolor{lightblue}7 & \cellcolor{lightblue}7 & \cellcolor{lightblue}4 & \cellcolor{lightblue}5 & \cellcolor{lightblue}2 &  \cellcolor{lightgray} & \cellcolor{lightgray} &
\cellcolor{lightgray} & \cellcolor{lightgray}  & \cellcolor{lightgray}  & \cellcolor{lightgray} & \cellcolor{lightpink}   & \cellcolor{lightpink}  & \cellcolor{lightpink}   & \cellcolor{lightpink}   \\
\hline
11 &\cellcolor{lightpink}  & \cellcolor{lightpink}  & \cellcolor{lightgray} &  \cellcolor{lightgray} & \cellcolor{lightblue}2 & \cellcolor{lightblue}12 & \cellcolor{lightblue}14 & \cellcolor{lightblue}15 & \cellcolor{lightblue}12 & \cellcolor{lightblue}2 & \cellcolor{lightblue}2 &  \cellcolor{lightblue}1& \cellcolor{lightblue}1 &  \cellcolor{lightgray} &  
\cellcolor{lightgray} & \cellcolor{lightgray} &
\cellcolor{lightgray} & \cellcolor{lightgray}  & \cellcolor{lightgray}& \cellcolor{lightgray} \\
\hline
12 & \cellcolor{lightpink}  & \cellcolor{lightpink}  & \cellcolor{lightpink} & \cellcolor{lightblue}1 & \cellcolor{lightgray} & \cellcolor{lightblue}1 & \cellcolor{lightblue}12 & \cellcolor{lightblue}15 & \cellcolor{lightblue}11 & \cellcolor{lightblue}9 & \cellcolor{lightblue}10 & \cellcolor{lightblue}7 & \cellcolor{lightblue}4 & \cellcolor{lightblue}1 & 
\cellcolor{lightblue}2 & \cellcolor{lightblue}1 &
\cellcolor{lightgray} & \cellcolor{lightgray} &
\cellcolor{lightgray} & \cellcolor{lightgray} \\
\specialrule{1pt}{0pt}{0pt}
13 & \cellcolor{lightpink}  & \cellcolor{lightpink}  & \cellcolor{lightpink}  & \cellcolor{lightpink} & \cellcolor{lightpink} & \cellcolor{lightgray} & \cellcolor{lightgray} & \cellcolor{lightblue}4 & \cellcolor{lightblue}13 & \cellcolor{lightblue}14 & \cellcolor{lightblue}15 & \cellcolor{lightblue}16 & \cellcolor{orange}6 &\cellcolor{orange}2 & \cellcolor{orange}2 &  \cellcolor{orange}1& &  & &   \\
\hline
14 & \cellcolor{lightpink} & \cellcolor{lightpink}  & \cellcolor{lightpink}  & \cellcolor{lightpink}  & \cellcolor{lightpink}  & \cellcolor{lightpink}  & \cellcolor{lightgray} & \cellcolor{lightgray} & \cellcolor{lightblue}5 & \cellcolor{lightblue}10 & \cellcolor{lightblue}14 & \cellcolor{lightblue}16 & \cellcolor{orange}5 &\cellcolor{orange}4 &  &  & &  & &   \\
\hline
15 & \cellcolor{lightpink} & \cellcolor{lightpink}  & \cellcolor{lightpink}  & \cellcolor{lightpink}  & \cellcolor{lightpink}  & \cellcolor{lightpink}  &  \cellcolor{lightpink} & \cellcolor{lightpink}& \cellcolor{lightgray} & \cellcolor{lightblue}1 & \cellcolor{lightblue}1 & \cellcolor{lightblue}6 & \cellcolor{orange}10 & \cellcolor{orange}7 & & & & & &   \\
\hline
16 &\cellcolor{lightpink} & \cellcolor{lightpink}  & \cellcolor{lightpink}  & \cellcolor{lightpink}  & \cellcolor{lightpink}  & \cellcolor{lightpink}  & \cellcolor{lightpink} & \cellcolor{lightpink}& \cellcolor{lightpink} & \cellcolor{lightblue}2 &  \cellcolor{lightgray} & \cellcolor{lightblue}1 &\cellcolor{orange}1  & \cellcolor{orange}6 & &\ & &  & &    \\
\hline
17 & \cellcolor{lightpink} & \cellcolor{lightpink}  & \cellcolor{lightpink}  & \cellcolor{lightpink}  & \cellcolor{lightpink}  & \cellcolor{lightpink}  & \cellcolor{lightpink} & \cellcolor{lightpink} & \cellcolor{lightpink} & \cellcolor{lightpink} &
\cellcolor{lightpink} & \cellcolor{lightgray} &
\cellcolor{lightgray} &  &  &  & & & &   \\
\hline
18 & \cellcolor{lightpink} & \cellcolor{lightpink}  & \cellcolor{lightpink}  & \cellcolor{lightpink}  & \cellcolor{lightpink}  & \cellcolor{lightpink}  & \cellcolor{lightpink} & \cellcolor{lightpink} & \cellcolor{lightpink} & \cellcolor{lightpink} &
\cellcolor{lightpink}  & \cellcolor{lightpink} &
\cellcolor{lightgray} & \cellcolor{lightgray} &  &  & & & &   \\
\hline
\end{tabularx}
\end{table}

\begin{theorem}
 There are 177 minor minimal non-apex graphs with 12 or fewer vertices.
 \end{theorem}
 
Pierce presented a  list of MMNA graphs in \cite{P}.  This list contains graphs with orders ranging from six to 16 and sizes ranging from 15 to 30. Pierce's counts are presented in  Table~\ref{tab:mmna1}. 
In this table, orange squares indicate that MMNA graphs of that order and size were known to exist, though additional graphs might not yet have been found; blue squares indicate that all MMNA graphs of that order and size were believed to have been found; gray squares indicate that no MMNA graphs of that order and size exist; and white squares indicate that, at the time of \cite{P}, no MMNA graphs of that order and size had been found. In total, Table~\ref{tab:mmna1} contains 157 MMNA graphs.

In Table \ref{tab:mmna2}, we present our own findings, using the same color scheme as Pierce's, except a light pink color marks those orders and sizes where the number of edges $e$ is either strictly greater than $4n-10$ (upper right corner) or strictly less than $\lceil 3n/2 \rceil$ (lower left corner).  

The search for MMNA graphs of order 13 or larger is impeded by the fact that a complete list of all maxnil graphs of orders \( n \ge 13 \) is not yet known. However, for orders \( 13 \le n \le 17 \) and sizes up to 26, we were able to use \textit{nauty} to generate all graphs with minimum degree~3. We then tested these graphs to determine whether they are MMNA. We describe our algorithms and code in the Section \ref{sectioncode}. Our findings are presented in Table \ref{tab:mmna2}.

\begin{theorem}
 There are 286 minor minimal non-apex graphs with 26 or fewer edges. 
 \end{theorem}

For graphs of order greater than 12 and size greater than 26, we employed  different (non-exhaustive) methods to find additional MMNA graphs. 
%We observed that all MMNA graphs found so far have the \textit{claw property}. 
%We say that a graph $G$ has the \textit{claw property} if every vertex of degree $3$ in $G$ has a discrete neighborhood; that is, none of its neighbors are adjacent to each other.
We used delta--wye and wye-delta operations on the nine known MMNA graphs with more than 26 edges (one of order 11 and eight of order 12) to create candidates for the MMNA property. It is known that neither of these two operations preserve apexity (or non-apexity), but it is often the case that these operations produce good candidates.  We also used nauty to search for non-triangular graphs with more than 12 vertices or 26 edges. 
For all graphs obtained in these ways, we verified whether they are MMNA.
This search yielded 44 new MMNA graphs, for a total of 339 graphs. 
These graphs are presented in .txt format at \cite{List}.

\section{Explanation of computer programs}
\label{sectioncode}
The computer search for MMNA graphs used nauty \cite{MP} together with code written in Python available in Appendix A and Appendix B \cite{PPP}. The search for up to 12 vertices relied on previously known maximal linklessly embeddable graphs (maxnil) of orders 12 and below. 

For each order up to 12, the search began by eliminating apex graphs from the list of maxnil graphs, followed by sorting the remaining non-apex graphs by size. 
For order 11, the sizes of maxnil non-apex graphs range from 27 to 34. 
For order 12, their sizes range from 28 to 38. 
The graphs of largest size were tested for minimality using \textit{ismmna.py}. 
Graphs determined to be MMNA were saved. 
Graphs that were not MMNA were further analyzed using nauty as follows: delete one edge in every possible way, discard any resulting graph whose minimum degree is less than 3, and then remove isomorphic duplicates. The resulting list was combined with the list of maxnil graphs of the same size, after which the new list was tested with \textit{isapex.py}. Apex graphs were discarded, and the non-apex graphs were tested for minimality with \textit{ismmna.py}, and so on. Each search terminated once sufficient edge removals caused all remaining graphs to become apex.
When we tested graphs of order 11, we allowed disconnected graphs to remain in the search. As a result, the final list included all linklessly embeddable MMNA graphs of order 11 and lower.
When we tested graphs of order 12, we removed disconnected graphs so that the search would not include graphs of order at most 11 that had already been found.

The program \textit{isapex.py} takes a .g6 file as input and checks whether each graph is apex. It uses the Python module NetworkX. For each graph, the program first checks planarity. If the graph is planar it is also apex.  If the graph is not planar, the program removes one vertex at a time and checks whether the resulting subgraph is planar, continuing until either a vertex is found whose removal yields a planar graph or all vertices have been tested. The program produces two output files: one containing the apex graphs and one containing the non-apex graphs.

The program \textit{ismmna.py} performs a similar analysis, also reading a .g6 file and using NetworkX. It is assumed that the input graphs are non-apex. For each graph, the program deletes one  edge at a time and checks whether the resulting subgraph is apex, using the same method as \textit{isapex.py}. If any one-edge-deletion subgraph is non-apex, the original graph is marked as not MMNA. If all one-edge-deletion subgraphs are apex, then the program  contracts one edge at a time and checks whether the resulting minors are apex. If any contracted minor is non-apex, the graph is marked as not MMNA. If all such minors are apex, the graph is saved as MMNA.

Both programs were designed to run in a multiprocessed manner. 
%Each reads a batch of graphs and processes each batch on a separate CPU core, improving efficiency. 
For orders~11 or less, excluding the nauty commands executed separately for each size, \textit{ismmna.py} and \textit{isapex.py} ran for a combined total of 1 hour and 7 minutes. For order~12, excluding the nauty commands, the two programs ran for a combined total of 130 hours, with \textit{isapex.py} running for a cumulative 75 hours and \textit{ismmna.py} running for a cumulative 48 hours. The program \textit{isapex.py} tested roughly three times as many graphs as \textit{ismmna.py}. These computations were performed on an Apple M4 Pro with 24~GB of memory and 14 cores.

For each order $13\le n\le 17$ and size $\lceil \frac{3n}{2} \rceil \le e \le 26$,  graphs of order $n$ and size $e$ with minimal degree at least 3 were generated with \textit{nauty}. These graphs were tested with \textit{isapex.py} and \textit{ismmna.py}. The two programs ran for a combined total of 340 hours. These computations were performed on two machines: a 2023 Mac Mini, with a 12-core Apple M2 pro processor and 32GB of RAM; and a 2019 HP Z4 machine, with a 10-core Intel i9-10900X CPU at 3.7 GHz, and 256 GB of RAM.
\section{Graphs of order 13}

The search for MMNA graphs of order 13 is not yet complete.  One can start by finding all maxnil graph of order 13. Or, one can use the observation that when an edge or a vertex is deleted from a MMNA graph, an apex graph is created.  We present three ideas below.
 \begin{enumerate}
\item Consider all maximal planar graphs of order 12. 
 There are 7,595 such graphs. 
Consider a thirteenth vertex that is connected to all 12 vertices of the planar graph (a cone). 
 Add an edge every possible way to each of the 7,595 graphs to obtain a set with 187,491 non-isomorphic graphs of order 13 \cite{MP}. 
 All MMNA graphs with 13 vertices are subgraphs of graphs within this set. 
\item Consider all maximal planar graphs of order 11. 
 There are 1,249 such graphs.  Consider a twelfth vertex that is connected to all 11 vertices of the planar graph. 
 Add a vertex of degree five in every possible way to each of the 1,249 graphs to obtain a set with 843,080 non-isomorphic graphs of order 13 \cite{MP}.
  All MMNA graphs with 13 vertices are subgraphs of the graphs within this set, since no MMNA graph has minimal degree six or higher. We are going to provide an argument for this fact in this section.
\item  Start by finding all maxnil graphs of order 13, a set likely less in size than those sets in items (a) and (b). Then use the same ideas as for order 12.
%This is due in part to the fact that a complete list of maxnil graphs of order 13 is not currently available. Without such a list, the number of graphs that must be tested exceeds our computational capabilities. Moreover, if the list turns out to be very large, this approach may prove untenable.
One can organize the search for maxnil graphs of order 13 by their minimal degree, using the results of \cite{NPP} and Theorem \ref{main13} presented in this section. 
In \cite{NPP}, Naimi and the first two authors proved that any maxnil graph of order $n$ and minimal degree two is obtained by taking the clique sum over $K_2$ of a maxnil graph of order $n-1$ with a $K_3$. 
They proved that any maxnil graph of order $n$ and minimal degree three is obtained by taking the clique sum over $K_3$ of a maxnil graph of order $n-1$ with a $K_4$. 
Using the know list of 6503 maxnil graphs of order 12, we've found that there are exactly 24 maxnil graphs of order 13 and minimal degree two, and exactly 33,250 maxnil graphs of order 13 and minimal degree three. 
One can obtain the maxnil graphs of order 13 and minimal degree four or five by adding a vertex of degree four or five respectively to edge subgraphs of the maxnil graphs of degree 12, and then selecting those that have minimal degree four or five, respectively. 
Further restrictions one can be imposed on the neighborhood of a degree four vertex in a maxnil graph \cite{NPP}. The list of graphs with minimum degree five will be far smaller because of the minimum degree requirement.
These operations would yield a complete list of maxnil graphs of order 13, since there is no such graph with minimum degree at least six, as shown in Theorem \ref{main13}.
\end{enumerate}

In \cite{Jo}, J{\o}rgensen proved that every simple graph of order at most 11 with $\delta(G) \ge 6$ contains a $K_6$ minor. Following this theorem, he conjectured that every 6-connected graph which contains no $K_6$ minor is 1-apex. 
In \cite{KNTW}, Kawarabayashi,  Norine, Thomas, and Wollan prove this conjecture for sufficiently large graphs. Little is known about small order graphs. 
In \cite{OP}, A. Pavelescu and Odeneal proved that any simple graph of order 12 and minimum degree at least six has a $K_6$ minor.  
In this section, we prove a similar result for order 13, which is a strengthening of  J{\o}rgensen's conjecture for this order.\\

\noindent \textbf{Theorem 1.} \textit{Let $G$ be a simple graph of order $13$ and minimal degree 6. Then either $G$ contains a $K_6$-minor or it is isomorphic to $K_1\ast Ic$.}\\

\noindent Here $Ic$ denotes the icosahedral graph, the only five regular planar graph of order 12. The graph $K_1\ast Ic$ obtained from $Ic$ by adding a vertex and connecting this vertex to all the vertices of $Ic$. We present some prerequisites for  the proof of Theorem \ref{main13}, beginning with a result of Mader for graphs of minimum degree at least 5 and an upper bound on the number of edges that we used earlier. We use $N_G(v)$ to denote the set of vertices adjacent to the vertex $v$ and $N_G[v]= N_G(v)\cup \{v\}.$ 
For a graph $G$, we denote the number of edges (the \textit{size}) of $G$ by $|G|$.

\begin{theorem}[Mader \cite{Ma}] Let $G$ be a simple graph and assume $\delta(G)\ge 5 $ . Then $G$ contains as a minor either $K_6^-$ ($K_6$ minus one edge), or the icosahedral graph.
\label{deg5}
\end{theorem}

\begin{theorem}[Mader \cite{Ma}] For every integer $2\le p \le 7$ and every simple graph $G$ of order $n\ge p-1$ that has no minor isomorphic to $K_p$, $G$ has at most $(p-2)n-$${p-1}\choose {2}$ edges.
\label{Mader}
\end{theorem}

\noindent For $n=6$, Theorem \ref{Mader} says that a graph of order $n$ and more than $4n-10$ edges has a $K_6$-minor. 
Graphs of order $n$ and size $4n-10$ which do not have $K_6$-minors have a specific structure, as it follows from the following theorem due to J{\o}rgensen.

\begin{theorem}[J{\o}rgensen \cite{Jor}] Let $p$ be a natural number, $5\le p \le 7$. Let G be a graph with
$n$ vertices and $(p-2)n-$${p-1}\choose {2}$ edges that is not contractible to $K_p$. Then
either G is an $MP_{p-5}$-cockade or $p =7$ and $G$ is the complete 4-partite
graph $K_{2,2,2,3}$.
\label{cockade}
\end{theorem}

When looking for $K_6$-minors for graphs of size $4n-10$, it is enough to study $MP_1-$cockades \cite{Jor}.

\begin{definition}
$MP_{1}$-cockades are defined recursively as follows:
\label{def:cockade}
  \begin{enumerate}
    \item $K_{5}$ is an $MP_1$-cockade and if $H$ is a 4-connected maximal 
          planar graph then $H*K_1$ is an $MP_1$-cockade. 
    \item Let $G_1$ and $G_2$ be disjoint $MP_1$-cockades, and let 
          $x_1,x_2, x_3$, and $x_{4}$ be the vertices of an induced $K_{4}$ subgraph in $G_1$ and let 
          $y_1,y_2, y_3$, and $y_{4}$ be the vertices of an induced  $K_{4}$ subgraph in $G_2.$ Then
          the graph obtained from $G_1 \cup G_2$ by identifying $x_j$ and 
          $y_j,$ for j=1,2,3,4, is an $MP_{1}$-cockade.
  \end{enumerate}
\end{definition}

The second item in the definition describes a clique sum of $G_1$ and $G_2$ over a $K_4$ subgraph. We use the following lemma, which is a corollary of Theorem \ref{cockade}.

\begin{lemma} Let $M$ denote a graph of order 12 and size 38, such that $\delta(M)\ge 5$. Assume that $M$  is not apex and has at most three vertices of degree 5. Then $M$ contains a $K_6$ minor.
\label{12cockade}
\end{lemma}

\begin{proof} By Theorem \ref{cockade} and Definition \ref{def:cockade}, $M$ must be the clique sum over $S\simeq K_4$ of two $MP_1$-cockades. If $M\setminus S$ has more than two connected components, then at least one of the connected components has at most 2 vertices. 
Since $\delta(M)\ge 5$, it follows that $M$ has a $K_6$ subgraph induced by the two vertices in a connected component and the vertices of $S$. Thus we may assume $M\setminus S$ has exactly two connected components. Let $Q_1$ and $Q_2$ denote the connected components of $M\setminus S$, with $|V(Q_1)|\le |V(Q_2)|$. Let $L_1$ denote the set of edges of $M$ with one endpoint in $Q_1$ and the other endpoint in $S$. Let $L_2$ denote the set of edges of $M$ with one endpoint in $Q_2$ and the other endpoint in $S$. As above, we may assume $|V(Q_1)|\ge 3.$\\

\noindent Assume $|V(Q_1)|=3$. If there exists $v\in V(Q_1)$ with  $\deg_M(v)>5$, then
$$16 \le \sum_{v\in V(Q_1)} \deg_M(v)=2|E(Q_1)|+|L_1|\le 3+ |E(Q_1)|+|L_1|.$$
The induced subgraph $\big<Q_1, S\big>_M$ has seven vertices and $|E(Q_1)|+|L_1|+|S| \ge 13+6=19$
edges. By Theorem \ref{Mader}, $\big<Q_1, S\big>_M$ (and $M$) contains a $K_6$-minor.\\
Assume $\deg_M(v)=5$ for all $v\in V(Q_1)$. 
Since $M$ had at most three vertices of degree five, it follows that $\deg_M(v)\ge 6$ for every vertex $v\in V(Q_2)$. In particular, every vertex of $Q_2$ is adjacent to at least a vertex of $S$. If $|E(Q_2)|=10$, then $Q_2\simeq K_5$, and $M$ has a $K_6$ minor obtained by contracting the subgraph $\big<Q_1,S\big>_M$ to a single vertex.
If $|E(Q_2)|\le 9,$ then 
$$30  \le \sum_{v\in V(Q_2)} \deg_M(v)=2|E(Q_2)|+|L_2|\le 9 + |E(Q_2)|+|L_2|.$$ The induces subgraph $\big<Q_2, S\big>_M$ has nine vertices and 
$|E(Q_2)|+|L_2|+|S| \ge 21+6=27$ 
edges. By Theorem \ref{Mader}, $\big<Q_2, S\big>_M$ contains a $K_6$-minor.\\
Assume $|V(Q_1)|=|V(Q_2)|=4$. Without loss of generality,  $Q_2$ has at most one vertex of degree $5$. Then 
$$23\le\sum_{v\in V(Q_2)} \deg_M(v) = 2|E(Q_2|+|L_2|\le 6+ |E(Q_2|+|L_2|.$$ 
The induced subgraph $\big<Q_2,S\big>_M$  has eight vertices and $|E(Q_2)|+|L_2|+6 \ge 17+6=23$ edges. 
By Theorem \ref{Mader}, $\big<Q_2, S\big>_M$ contains a $K_6$-minor.\\

\end{proof}

\begin{proof}[Proof of Theorem \ref{main13}]

Let $G$ denote a simple graph of order 13 and minimal degree $\delta(G)\ge 6$. If $G$ is apex, then there exists $v\in V(G)$ such that $G\setminus\{v\}$ is planar of minimal degree at least $5$. On one hand, by Euler's theorem for planar graphs, $G\setminus\{v\}$ has at most $3\times 12 -6 =30$ edges. On the other hand, since $\delta(G\setminus\{v\})\ge 5$, $G\setminus\{v\}$ has at least 30 edges. It follows that $G\setminus\{v\}$ is 5-regular. The only 5-regular planar graph with 12 vertices is the icosahedral graph $Ic$. Since $\delta(G)\ge 6$, it follows $v$ is adjacent to all vertices in $G\setminus \{v\}$ and $G\simeq K_1 \ast Ic$. For the rest of the proof we shall assume that  $G$ is not apex.

\noindent Since $\delta(G)\ge 6$, $|E(G)|\ge 39$. By Theorem \ref{Mader}, if the size of $G$ is at least 43, then $G$ contains a $K_6$ minor. We organize the prove Theorem \ref{main13} by considering the size of $G$, $39 \le |E(G)|\le 42$. \\

\noindent {\bf Case 1} Assume $|E(G)|=42$. Since $G$ is not apex, by Theorem \ref{cockade}, it follows that $G$ is a clique sum over $S\simeq K_4$ of two $MP_1$-cockades.
If any of the connected components of $G\setminus S$ has at most 3 vertices, since $\delta(G)\ge 6$, then any such connected component must have exactly 3 vertices and that component together with $S$ must induce a $K_7$ subgraph of $G$. It follows that $G\setminus S$ has exactly two connected components, of order 4 and 5, respectively. Let $Q$ denote the connected component of $G\setminus S$ of order 4. Let $L$ denote the set of edges of
$M$ with one endpoint in $Q$ and the other endpoint in $S$. Then 
$$24\le\sum_{v\in V(Q)} \deg_M(v) = 2|E(Q|+|L|\le 6+ |E(Q|+|L|.$$ 
The induced subgraph $\big<Q,S\big>_M$  has eight vertices and $|E(Q)|+|L|+6 \ge 18+6=24$ edges. 
By Theorem \ref{Mader}, $\big<Q, S\big>_M$ contains a $K_6$-minor.\\

\noindent {\bf Case 2} Assume $|E(G)|=41$. The degree of any vertex of $G$ is at most 10. Let $V(G)=\{v_1, v_2, \ldots, v_{13}\}$ and assume $v_1v_2\in E(G)$. Let $G'$ be the minor of $G$ obtained by contracting the edge $v_1v_2$.  If $v_1$ and $v_2$ share less that two neighbors, then $G'$ has 12 vertices and at least 39 edges, and by Theorem \ref{Mader}, $G'$ contains  a $K_6$-minor. 
If $v_1$ and $v_2$ share exactly two neighbors, then $G'$ has 38 edges and at most three vertices of degree 5 (the possible such vertices are the vertex obtained through the contraction of $v_1v_2$, and the two common neighbors of $v_1$ and $v_2$). 
If $G'$ is apex, then $G'$ is maximal apex and it has a vertex of degree 11. This vertex can only be obtained through the contractions of $v_1v_2$ and only if $\deg(v_1)+\deg(v_2)=15.$ If $G'$ is not apex, by Lemma \ref{12cockade} it contains a $K_6$ minor.\\

\noindent \textit{Assume $G$ has a vertex of degree at least 9.}
The  possible degree sequences for $G$ are $(6, 6, 6, 6, 6, 6, 6, 6, 6, 6, 6, 6, 10)$ and $(6, 6, 6, 6, 6, 6, 6, 6, 6, 6, 6, 7, 9)$.  Let $v_{13}\in V(G)$ such that $\deg(v_{13})\ge 9$.  There exists $v_1\in V(G)$ with $\deg(v_1)=6$ and $v_1v_{13}\notin E(G)$.
Let $N=N_G(v_1)=\big <v_2, v_3,\ldots, v_7 \big>_G$. Let $H=\big <v_8,\ldots,v_{13}\big >_G$ and let $L$ denote the set of edges of $G$ with one endpoint in $N$ and the other in $H$. 
 If $v_1$ and $v_2$ share two neighbors, since $\deg(v_1)+\deg_G(v_2)<15$, as described above, $G$ contains a $K_6$-minor.
We assume that $v_1$ and $v_2$ share at least three neighbors. 
This implies that $E(N)\ge 9$.
Adding degrees in both $N$ and $H$, respectively,  we get
\[\sum_{v\in V(N)} \deg_G(v)= 2|E(N)|+|L|+6\le 6 \times 5+7=37,\]
\[\sum_{v\in V(H)} \deg_G(v)= 2|E(H)|+|L|\ge 6 \times 5+9=39.\]
Combining the inequalities gives $|E(H)|\ge |E(N)|+4\ge 13$. Since $\delta(G)\ge 6$, every vertex of $H$ must be adjacent to a vertex of $N$, and contracting all six edges incident to $v_1$ gives a minor of $G$ of order 7 and size at least 19. By Theorem \ref{Mader}, this minor has a $K_6$-minor.\\

\noindent \textit{Assume that the maximum degree of $G$ is 8}. The possible degree sequences of $G$ are  $(6, 6, 6, 6, 6, 6, 6, 6, 6, 6, 6, 8, 8),$ and  $(6, 6, 6, 6, 6, 6, 6, 6, 6, 6, 7, 7, 8).$
%In either case, if any of the vertices of degree at least 7 are adjacent, deleting the edge incident to both of them gives a graph of order 13, size 40, and minimum degree 6. We address this configuration in the next case. 
Without loss of generality, assume that $deg_G(v_{13})=8$, $deg_G(v_1)=6$ and $v_1v_{13} \notin E(G)$. Let $N$, $H$, and $L$ be as above. As above, we can assume $|E(N)|\ge 9.$
%Let $N=N_G(v_1)=\big < v_2,...,v_7\big >_G$,  let $H=\big < v_8,...,v_{13}\big >_G$ be the subgraph induced by the non-neighbors of $v_1$, and let $L$ denote the set of edges with one endpoint in $N$ and the other in $H$. 
If $H$ contains another vertex of degree at least 7 in $G$ besides $v_{13}$, then 

\[ \sum_{v\in V(N)} \deg_G(v)= 2|E(N)|+|L|+6\le 37,\]
\[\sum_{v\in V(H)} \deg_G(v) = 2|E(H)|+|L|\ge 6\times 4 +7+8=39.\]

\noindent Combining the inequalities gives $|E(H)|\ge |E(N)|+4\ge 13$. Contracting all six edges incident to $v_1$ gives a minor of order 7 and size at least 19, which must contain a $K_6$ minor by Theorem \ref{Mader}. 

Assume $H$ contains no other  vertex of degree at least 7 besides $v_{13}$. Then we have $2|E(N)|+|L|=32$, $2|E(H)|+|L|=38$, and $|E(H)|=|E(N)|+3$. If  $|E(N)|>9$ then $|E(H)|\ge 13$ and a $K_6$ minor follows as above. 
If $|E(N)|=9$, then $|L|=14$, $|E(H)|=12$, and $N$ must be 3-regular.
The subgraph $N$ isomorphic to either $K_{3,3}$, or the prism graph. We look at the two cases below. Since $H$ has size 12 and order 6, is follows that $H$ is 2-connected.\\ 

\noindent Assume $N$ is isomorphic to $K_{3,3}$.  Since $deg_G(v_{13})=8$, $v_{13}$ is adjacent to three or more vertices of $N$. When an edge between $v_{13}$ and one of its neighbors in $N$ is contracted at least one new edge between the vertices of $N$ is created.  The subgraph $H\setminus \{v_{13}\}$ is connected and each vertex of $N$ connects to at least one vertex of $H\setminus \{v_{13}\}$. Contracting the edges of $H\setminus \{v_{13}\}$ and then contracting an edge between $v_{13}$ and one of its neighbors in $N$ gives a minor that contains a subgraph isomorphic to the graph in Figure \ref{k33+e}. Contracting the edges $cd$ and $v_1e$ gives a $K_6$-minor.

\begin{figure}[htpb!]
\begin{center}
\begin{picture}(150, 110)
\put(0,0){\includegraphics[width=2in]{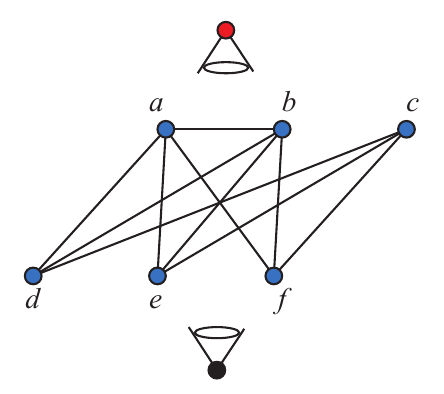}}
\put(70,-5){$v_1$}
\put(73,125){$v_8$}
\end{picture}
\caption{Minor of $G$. The vertices $v_1$ and $v_8$ are adjacent to $a, b, c, d, e,$ and $f$.}
\label{k33+e}
\end{center}
\end{figure} 

\noindent Assume that $N$ is isomorphic to the prism graph, with labels as in Figure \ref{prism}.
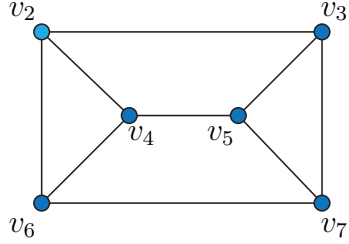
\begin{figure}[htpb!]
\begin{center}
\begin{picture}(150, 90)
\put(0,0){\includegraphics[width=2in]{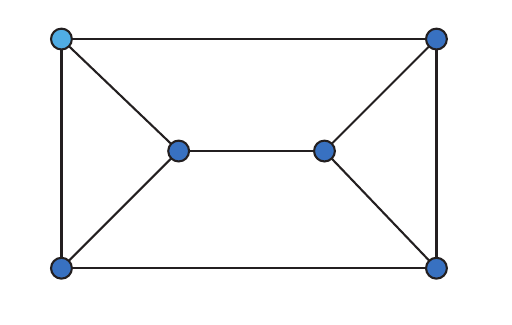}}
\put(5,87){$v_2$}
\put(123,87){$v_3$}
\put(5,5){$v_6$}
\put(123,5){$v_7$}
\put(50,40){$v_4$}
\put(80,40){$v_5$}
\end{picture}
\caption{The neighborhood of $v_1$.}
\label{prism}
\end{center}
\end{figure}

\noindent For all  $v\in V(H)$, $H\setminus \{v\}$ is connected, and every vertex of $N$ is adjacent to at least one vertex of $H\setminus \{v\}$. If there exists $v\in V(H)$ such that the set of neighbors of $v$ in $N$ does not induce a clique in $G$, then contracting an edge connecting $v$ to one of its neighbors in $N$ and then contracting all edges of $H\setminus \{v\}$ gives a minor that contains a subgraph isomorphic to the graph in Figure \ref{prism-coned}.  This graph has a $K_6$ minor obtained by contracting the edges $ae$ and $v_1c$.

\begin{figure}[htpb!]
\begin{center}
\begin{picture}(150, 110)
\put(0,0){\includegraphics[width=2in]{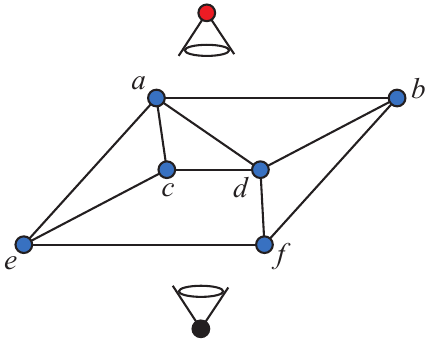}}
\put(64,-5){$v_1$}
\put(67,121){$v_8$}
\end{picture}
\caption{Minor of $G$. The vertices $v_1$ and $v_8$ are adjacent to $a, b, c, d, e,$ and $f$.}
\label{prism-coned}
\end{center}
\end{figure} 

\noindent Assume that for all vertices $v\in V(H)$, the set of neighbors of $v$ in $N$ induces a clique in $G$. In particular, $v_{13}$ is adjacent  to three vertices of $N$ and  $deg_H(v_{13})=5$. 
Without loss of generality, assume that $v_{13}$ is adjacent to $v_2$, $v_4$, and $v_6$, 
%and \textcolor{red}{$deg_G(v_2)=deg_G(v_4)=deg_G(v_6)=6$ why? can we have a degree 7 vertex? I don't think this reference to degrees is needed}, 
as labeled in Figure \ref{prism}.  As $deg_H(v_{13})=5$ and $|E(H)|=12$, by \cite{atlas}, it follows that $H$ is isomorphic to one of the four graphs in Figure \ref{6.12.5}.

\begin{figure}[htpb!]
\begin{center}
\begin{picture}(480, 110)
\put(0,0){\includegraphics[width=6.5in]{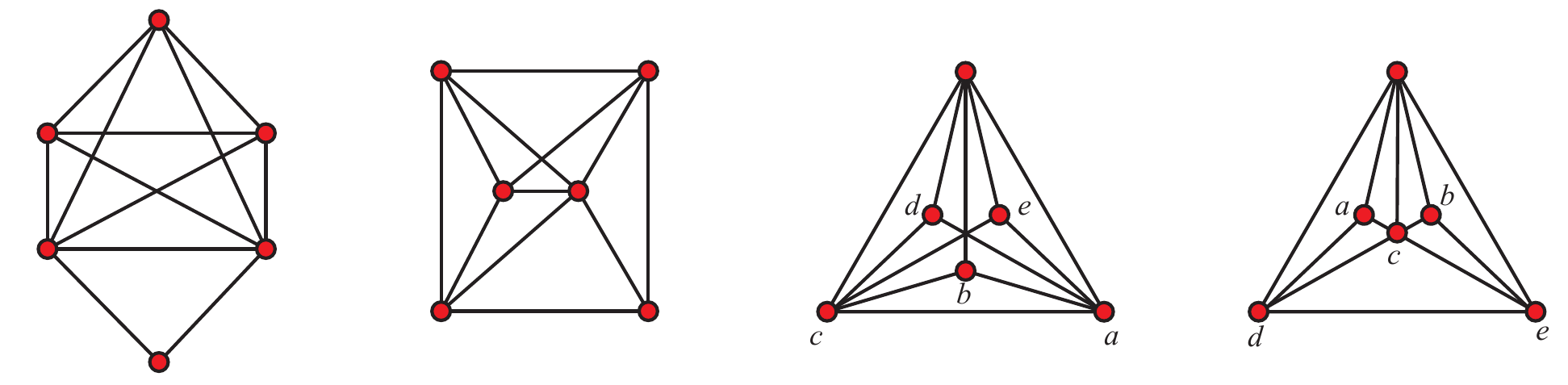}}
\put(278,104){$v_{13}$}
\put(408,104){$v_{13}$}
\end{picture}
\caption{Graphs of order 6 and size 12 that have a vertex of degree 5.}
\label{6.12.5}
\end{center}
\end{figure} 

\noindent The first two graphs in Figure \ref{6.12.5} have $K_5$ minors. If $H$ is isomorphic to either of them, contracting the edges of $N_G[v_1]$ gives a minor of $G$ that has a $K_6$-minor. 
If $H$ is isomorphic to the third graph in Figure \ref{6.12.5}, consider the neighbors in $N$ of the vertices labeled $b, d,$ and $e$. The neighbors in $N$ of each of these vertices must induce a clique of order 3. Since $N$ contains only two induced cliques of order 3, we may assume that $b$ and $d$ are adjacent to the same clique. If $b$ and $d$ are adjacent to $v_2, v_4,$ and $v_6$, then contracting the edge $ab$ yields a minor $G'$ of $G$ whose induced subgraph $\big< v_2, v_4, v_6, b, d, v_{13} \big>_{G'}$ is isomorphic to $K_6$. If $b$ and $d$ are adjacent to $v_3, v_5,$ and $v_7$, then contracting the edge $ab$ together with the edges of $\big< v_2, v_4, v_6, v_{13} \big>_G$ yields a minor $G'$ of $G$ whose induced subgraph $\big< v_3, v_5, v_7, b, d, v_{13} \big>_{G'}$ is isomorphic to $K_6$.
If $H$ is isomorphic to the fourth graph in Figure \ref{6.12.5}, consider the two vertices of degree 3 in $H$, labeled $a$ and $b$. They each must be adjacent to the three vertices of a clique in $N$. If one of $a$ or $b$ (say $b$) is adjacent to $v_2$, $v_4$, and $v_6$, contracting the edges in the graph $\big <v_1, v_3, v_5, v_7, a, c, d, e\big >_G$ gives a $K_6$-minor.
 If both $a$ and be $b$ are adjacent to $v_3$, $v_5$, and $v_7$, contracting the edge $ac$ and the edges of $\big <v_1, v_2, v_4, v_6, d, e, v_{13}\big >_G$ gives a $K_6$-minor. \\
 
\noindent \textit{Assume the maximum degree of $G$ is $7$.} The degree sequence of $G$ is ($6,6,6,6,6,6,6,6,6,$ $7,7,7,7$). Assuming that none of the degree 7 vertices are adjacent (we can delete one such edge and obtain a graph with 40 edges), there are exactly 28 edges connecting vertices of degree 6 to vertices of degree 7. By the Pigeonhole principle, there is a vertex of degree 6, say $v_1$, which is adjacent to all the vertices of degree 7, say $a$, $b$, $c$, and $d$. Let $N_G(v_1)=\big <v_2, v_3, a, b, c, d \big >_G$. The adjacent vertices $v_1$ and $a$ have at most two common neighbors, and $\deg(v_1)+\deg(a)=13<15.$ Following the observations at the beginning of the proof, $G$ contains a $K_6$-minor.\\

\noindent {\bf Case 3} Assume $39\le |E(G)|\le 40$. Since $\delta(G)\ge 6,$  $G$ is connected.  If $v\in V(G)$ such that $G\setminus\{v\}$ is disconnected, then each connected component of $G\setminus\{v\}$ has six vertices. Since  $\delta(G)\ge 6,$ and $\big <S,v\big >_G\simeq K_7$ for each connected component $S$ of $G\setminus\{v\}$, and thus $G$ has a $K_6$ minor. We can assume $G$ is 2-connected. 
%Stricter restrictions on the vertex connectivity of such a graph can be derived while trying to avoid a $K_6$ minors. 
There are exactly 367,860 six-regular, 2-connected graphs of order 13 and exactly 23,489,426 2-connected graphs of order 13, size 40, and minimum degree 6. These were generated using the nauty program \cite{MP}. Using parallel Python computations employing the code in the \textit{minor miner} tool described in \cite{CMR}, we found all these graphs have a $K_6$ minor. It would still be interesting finding a computer-free argument for sizes 39 and 40.

\end{proof}
\noindent We conclude with two open questions.\\
\noindent \textbf{Question 1.} What is the complete set of maxnil graphs of order 13?

\noindent \textbf{Question 2.} What is the complete set of MMNA graphs?

\bibliographystyle{amsplain}

\end{document}